\documentclass[12pt,fleqn] {article}
\usepackage{amssymb,amsmath,amsthm,setspace}
\usepackage{graphicx}
\usepackage{authblk}
\usepackage{natbib}
\newcommand{\R}{\mathbb{R}}

\newcommand{\adb}{\allowdisplaybreaks}
\newcommand{\bs}{\boldsymbol}

\newtheorem{Theorem}{Theorem}

\newtheorem{Corollary}{Corollary}

\newcommand{\inv}{\frac{1}}
\newcommand{\bosy}{\boldsymbol}
\date{21-02-2015}
\title{Tail dependence convergence rate for the bivariate skew normal under the equal-skewness condition}
\author[1]{Thomas Fung\thanks{Corresponding Author. Honorary Associate, University of Sydney. Email address: thomas.fung@mq.edu.au.}}
\author[2]{Eugene Seneta}
\affil[1]{Department of Statistics, Macquarie University, NSW 2109, Australia}
\affil[2]{School of Mathematics and Statistics,  University of Sydney, NSW 2006, Australia}

\begin{document}
\date{\today}

\maketitle
\begin{abstract}
\noindent We derive the rate of decay of the tail dependence of the bivariate
skew normal distribution under the equal-skewness condition $\theta_1=\theta_2, = \theta$, say.  The rate of convergence depends on whether $\theta >0$ or $\theta <0$. The latter case gives rate asymptotically identical with the case $\theta =0$. The asymptotic behaviour of the quantile function for the univariate skew normal is part of the theoretical development.
\end{abstract}



\section{Introduction}
The coefficient of lower tail dependence of a random vector $\textbf{X} =
(X_1,X_2)^T$ with marginal inverse distribution function $F_1^{-1}$ and
$F_2^{-1}$ is defined as
\begin{equation}
\lambda_L = \lim_{u \rightarrow 0^+}\lambda_L(u), \quad \text{where} \quad \lambda_L(u) = P(X_1 \leq F_1^{-1}(u) | X_2
\leq F_2^{-1}(u)).\label{defn:tail dependence}
\end{equation}
$\textbf{X}$ is said to have asymptotic lower  tail dependence if $\lambda_L$ exists and is positive. If $\lambda_L=0$, then $\textbf{X}$ is said to be
asymptotically independent in the lower tail. 
This quantity provides insight on
the tendency for the distribution to generate joint extreme event since it
measures the strength of dependence (or association) in the lower tails of a
bivariate distribution.  If the marginal distributions of these random variables are
continuous, then from (\ref{defn:tail dependence}), it follows that
$\lambda_L(u)$ can be expressed in terms of  the copula of
$\textbf{X}$, $C(u_1,u_2)$, as
\begin{equation}
\lambda_L(u) = \frac{P(X_1\leq F_1^{-1}(u), X_2\leq F_2^{-1}(u))}{P(X_2 \leq F_2^{-1}(u))} =  \frac{C(u,u)}{u}.
\label{defn:tail dependence 3}
\end{equation}
 \citet{ramo2009}, continuing the work of \citet{ledf1997}, studied intensively a family of bivariate distributions (which they characterised) which satisfied in particular the condition
\begin{equation}
\lambda_L(u) = u^{\inv{\alpha}-1} L(u) \label{RL condition}
\end{equation}
where $L(u)$ is a slowly varying function (SVF) as $u\to 0^+$, and $\alpha\in (0,1]$, so that, in fact, the value of $\alpha$ could be used for comparison of the degree of tail dependence structure between members of the family. The standard bivariate extreme value models correspond to $\alpha=1.$ 

\citet{hua2011} developed this idea further and defined $\kappa = 1/\alpha$ in (\ref{RL condition}) as the (lower) tail order of a copula. The case $1 < \kappa<2$ is termed as intermediate tail dependence as it represents the copula has some level of positive dependence in the tail but not as strong as tail dependence with $\lambda_L=0$. The tail order $\kappa$ can be used to assess tail dependence strength when $\lambda_L=0$.

The standard bivariate normal with correlation coefficient $-1<\rho<1$ corresponds to $\alpha = \frac{1+\rho}{2}$ in (\ref{RL condition}), and hence is an instance of intermediate tail dependence.  A  recent  manifestation of this  more or less known  result  is in \citet{fung2011}, where it is shown that 
\begin{equation*}
L(u) \sim 2 \sqrt{\frac{1+\rho}{1-\rho}} (-4\pi \log u)^{-\frac{\rho}{1+\rho}}, \quad \text{as $u\to 0^+$}.
\end{equation*}

The bivariate skew normal distribution was introduced in Azzalini and Dalla Valle (1996) (which is discussed further in Azzalini and Capitanio (1999)). A random vector $\textbf{X}$ is said to have a bivariate skew normal distribution distribution, denoted as $\textbf{X} \sim SN_2(\bs{\theta}, R)$, 
if the probability density of $\textbf{X}$ is  
\begin{equation}
f(\textbf{x}) = 2\phi_2(\textbf{x},R)\Phi(\bs{\theta}^T\textbf{x}), \label{pdf:SN_n}
\end{equation}
where $\phi(\cdot,R)$ is density of a bivariate normal distribution with mean $\textbf{0}$ and correlation matrix $R$ and $\Phi(\cdot)$ is the cdf of a univariate standard normal distribution. The correlation matrix $R$ and skew vector $\bs{\theta}$ are defined as $\left(\begin{smallmatrix} 1 & & \rho \\ \\ \rho & & 1 \end{smallmatrix}\right)$, with $-1<\rho<1$ and $\bs{\theta} =(\theta_1,\theta_2)^T \in \R^2$ respectively. Obviously, the (symmetric) normal is obtained as special case when $\bs{\theta}=0$. 

The results of Lysenko, Roy and Waeber (2009), Bortot (2010) and Padoan  (2011) show  that the skew normal distribution is tail independent, that is: $\lambda_L = \lim_{u \rightarrow 0^+}\lambda_L(u) =0.$ 

The focus of this current note is  thus to consider (\ref{RL condition}) in the setting of the skew normal distribution.  

We were motivated to do this not only  by  interest in generalizing to bivariate skew normal   the  result for the bivariate normal, but also by the following. 
\citet{fung2014} chose to view results such as  (\ref{RL condition}) as a rate of convergence result to the limit value
 $\lambda_L=0$. The skew $t$ distribution  is tail dependent i.e. $\lambda_L>0$ and \citet{fung2014} obtained for it
the rate of convergence result:
\begin{equation}\label{tcvgce}
\lambda_L(u) - \lambda_L   = \mathcal{K}(\eta, R,\bs{\theta} ) u^{\frac{2}{\eta}}+O(u^{\frac{4}{\eta}})
\end{equation}
as $u\to 0^+$, where $\mathcal{K}(\eta, R,\bs{\theta})$ is a constant which  depends on the distribution's degrees of freedom $\eta$, scale matrix $R$ and skewness vector $\bs{\theta}. $ This generalised the results of \citet{mann2011} and \citet{chic2012}, who had taken the rate of convergence standpoint and showed its practical importance, for the (symmetric) Student's $t$ distribution.  

Now,  this skew $t$ approaches the skew normal and there are cases where $\lambda_L$ and $\mathcal{K}(\eta,R,\bs{\theta})\to 0$ as $\eta \to \infty$. But it is clear that the expression for $\lambda_L(u)$ deriving from  (\ref{tcvgce}),  in the limit does not seem to  provide useful information about an expression for convergence rate  for the limit. Our treatment of convergence rate for  the  skew $t$ case had worked in a unified way  when $\theta_1$ was not necessarily the same as $\theta_2$. For the skew-normal we confined ourselves for this paper to equi-skewness: that is $\theta_1 = \theta_2, = \theta.$   Even so,  we found that our approach provided  qualitatively and quantitavely different results  in the cases $ \theta > 0$ and $\theta <0$, the latter case  being asymptotically (apart from a constant multiplier) identical  to the  symmetric bivariate normal case $\theta = 0.$ The two cases required quite different approaches. 

Upper tail dependence behaviour is expressed  immediately from our result on lower tail dependence  behaviour, using the device  in  \citet{fung2014}, Section 5. 

The remainder of this  paper is set out as follows. In Section 2, we
derive the asymptotic behaviour of the quantile function for the skew normal which is needed for our subsequent proof.  In Section 3, we derive the rate of convergence in the form of (\ref{RL condition}) for the skew normal distribution.

\section{Asymptotic behaviour of the quantile function}
Under the equal skewness condition i.e. $ \theta_1=\theta_2 =\theta$, both $X_1$ and $X_2$ will have the same marginal distribution. Without loss of generality, we shall focus on $X_1$. The marginal density for $X_1$ is 
\begin{equation}
f_{X_1}(x_1) = 2 \phi(x_1) \Phi(\lambda x_1), \quad \text{for $x_1\in \R$}, \label{pdf:SN_1}
\end{equation}
where $\phi(\cdot)$ and $\Phi(\cdot)$ are the pdf and cdf of the univariate standard normal and 
\begin{equation}
\lambda = \frac{\theta(1+\rho)}{\sqrt{1+\theta^2(1-\rho^2)}}. \label{lambda}
\end{equation}
 This means that both $X_1$ and $X_2$ have a univariate skew normal distribution with skewness parameter $\lambda$ i.e. $X_i \sim SN(\lambda)$, $i=1,2$.

Using Lemma 2 of Capitanio (2010), we have the inequalities for $P(X_1\leq z)$, for $z<0$,
\begin{eqnarray}
\notag && \inv{\pi} e^{-\inv{2}(1+\lambda^2)z^2}\left[ \inv{\lambda(1+\lambda^2)}z^{-2} - \left(\frac{2}{\lambda(1+\lambda^2)} + \inv{\lambda^3(1+\lambda^2)}\right) z^{-4}\right] \\
&<& P(X_1 \leq z)  < \inv{\pi} e^{-\inv{2}(1+\lambda^2)z^2}\left[ \inv{\lambda(1+\lambda^2)}z^{-2} \right] \label{SN tail behaviour lambda>0}
\end{eqnarray}
when $\lambda>0$ and 
\begin{eqnarray}
\notag && \frac{2}{\sqrt{2\pi}}e^{-z^2/2}\Biggl[ |z|^{-1}-\sqrt{\frac{2}{\pi}}\inv{|\lambda|(1+\lambda^2)}z^{-2}e^{-z^2\lambda^2/2} - |z|^{-3}\Biggr] \\
\notag &<& P(X_1\leq z) < \frac{2}{\sqrt{2\pi}}e^{-z^2/2}\Biggl[ |z|^{-1}-\sqrt{\frac{2}{\pi}}\inv{|\lambda|(1+\lambda^2)}z^{-2}e^{-z^2\lambda^2/2} \\
&& \quad + \sqrt{\frac{2}{\pi}} |z|^{-4}e^{-z^2\lambda^2/2}\left(\frac{2}{|\lambda|(1+\lambda^2)^2}+\inv{|\lambda|^3(1+\lambda^2)}\right)\Biggr] \label{SN tail behaviour lambda<0}
\end{eqnarray}
when $\lambda<0$. This means that 
\begin{equation}
F_{1}(z) = P(X_1\leq z) \sim \begin{cases} 
\inv{\pi\lambda(1+\lambda^2)}|z|^{-2} e^{-\inv{2}(1+\lambda^2)z^2}, &\text{for $\lambda>0$}; \\
\sqrt{\frac{2}{\pi}} |z|^{-1} e^{-z^2/2}, \quad \text{for $\lambda<0$}
 \end{cases}\quad \text{as $z\to -\infty$}. \label{SN tail behaviour first order}
 \end{equation}
Recall, for comparison, the cdf of standard normal has the following asymptotic behaviour (see for instance Feller (1968) Chapter VII Lemma 2): 
 \begin{equation}
\Phi(z) \sim \inv{\sqrt{2\pi}} |z|^{-1} e^{-\inv{2}z^2}, \quad \text{as $z\to -\infty$} \label{normal cdf expansion}
\end{equation}
and neither of the expressions on the right-hand side of (\ref{SN tail behaviour first order}) reduces to (\ref{normal cdf expansion}) by letting $\lambda \to 0$.
 
The quantile results are summarised into the following theorem. The asymptotic expressions on the right are given in form convenient for the sequel.

\begin{Theorem}
Let $X_1\sim SN(\lambda)$, then 
\begin{equation*}
F_{1}^{-1}(u) \sim y(u) = \begin{cases}
 -\sqrt{ - \frac{2}{1+\lambda^2}\log(- 2\pi\lambda u \log(2\pi \lambda u))}
, & \text{if $\lambda>0$;}\\
 -\sqrt{-2\log\left(\frac{u}{2}\sqrt{-4\pi \log(\frac{u}{2}\sqrt{2\pi}})\right)}
, & \text{if $\lambda<0$,}
\end{cases} \quad \text{as $u\to 0^+$}.
\end{equation*}
\end{Theorem}

\begin{proof}
In order to find the asymptotic behaviour of the quantile functions $F^{-1}_{1}(u)$ as $u\to 0^+$, we shall use Theorem 1 of Fung and Seneta (2011), which requires to find $-y(u)$, a slowly varying function (SVF) as $u\to 0^+$, such that $F_1(y(u))/u \to 1$ as $u\to 0^+$. For then, according to that theorem, $F_1^{-1}(u) \sim y(u)$ as $u\to 0^+$. To proceed, we first note that both asymptotic expression on the right of (\ref{SN tail behaviour first order}) are of the form 
\begin{equation}
u = a|g(u)|^{b} e^{-c|g(u)|^d} , \label{tail behaviour general form}
\end{equation} 
where $a, c, d>0$, $b<0$. To solve for $g(u)$ in (\ref{tail behaviour general form}) as $u\to 0^+$ we use the Lambert $W$ function. 
The function itself is defined as the solution $w$ to $z = we^{w}$ for $z>0$: that is, as the inverse function of the positive continuous and increasing function $xe^x$, $x>0$. Moreover, the asymptotic behaviour of Lambert $W$ function is given by
\begin{equation}
W(z) = \log z - \log \log z + O(\frac{\log\log z}{\log z}),  \label{asym Lambert W}
\end{equation}
as $z\to \infty$ so $W(z) \to \infty$ as $z\to \infty$; see Corless \emph{et al.} (1996). From (\ref{tail behaviour general form}),
\begin{eqnarray*}
u = a|g(u)|^{b} e^{-c|g(u)|^d}  \quad \Rightarrow  \quad \frac{cd}{|b|}|g(u)|^d e^{\frac{cd}{|b|}|g(u)|^d} = \frac{cd}{|b|}\left(\frac{a}{u}\right)^{d/|b|}, \label{tail behaviour general form 2}
\end{eqnarray*}
so the LHS of the last expression is in the form of $we^{w}$. Thus
\begin{eqnarray}
\frac{cd}{|b|}|g(u)|^d = W\left(\frac{cd}{|b|}\left(\frac{a}{u}\right)^{\frac{d}{|b|}}\right) \label{result in Lambert W form}
\end{eqnarray}
where $W(\cdot)$ is the Lambert $W$ function. 
 As $u\to 0^+$, $u^{-\frac{d}{|b|}} \to \infty$ and we can combine (\ref{asym Lambert W}) with (\ref{result in Lambert W form}) to get
\begin{eqnarray}
\notag \frac{cd}{|b|}|g(u)|^d &=& \log\left(\frac{cd}{|b|}\left(\frac{a}{u}\right)^{\frac{d}{|b|}}\right) - \log \log \left(\frac{cd}{|b|}\left(\frac{a}{u}\right)^{\frac{d}{|b|}}\right) + O(\left(\frac{\log \log \left(\frac{cd}{|b|}\left(\frac{a}{u}\right)^{\frac{d}{|b|}}\right)}{\log\left(\frac{cd}{|b|}\left(\frac{a}{u}\right)^{\frac{d}{|b|}}\right)}\right)\\
\Rightarrow \quad |g(u)|& =& \left\{ \frac{|b|}{cd}\left[ \log\left(\frac{\frac{cd}{|b|}\left(\frac{a}{u}\right)^{\frac{d}{|b|}}}{\log \left(\frac{cd}{|b|}\left(\frac{a}{u}\right)^{\frac{d}{|b|}}\right)}\right) + O\left(\frac{\log \log \left(\frac{cd}{|b|}\left(\frac{a}{u}\right)^{\frac{d}{|b|}}\right)}{\log\left(\frac{cd}{|b|}\left(\frac{a}{u}\right)^{\frac{d}{|b|}}\right)}\right)\right] \right\}^{\inv{d}}, \label{solution to Lambert W function}
\end{eqnarray}
as $u\to 0^+$.

Comparing the expression on the right of (\ref{SN tail behaviour first order}) with (\ref{tail behaviour general form}), we have 
$a = \inv{\pi\lambda(1+\lambda^2)}$, $b = -2$, $c = \inv{2}(1+\lambda^2)$ and $d=2$ so that $\frac{cd}{|b|} = \frac{1+\lambda^2}{2}$ for $\lambda>0$; and $a= \sqrt{2/\pi}$, $b= -1$, $c = 1/2$ and $d= 2$ so that $\frac{cd}{|b|} = 1$ for $\lambda<0$. Substitute these constants into (\ref{solution to Lambert W function}) to get {\adb
\begin{eqnarray}
\notag g(u) &=& - \Biggl\{ \frac{2}{(1+\lambda^2)}\Biggl[ \log\left(\frac{\frac{(1+\lambda^2)}{2}(\inv{\pi\lambda(1+\lambda^2) u})}{\log\left(\frac{(1+\lambda^2)}{2}(\inv{\pi\lambda(1+\lambda^2) u})\right)}\right)\\
\notag && \quad + O\left(\frac{\log\log\left(\frac{(1+\lambda^2)}{2}(\inv{\pi\lambda(1+\lambda^2) u})\right)}{\log\left(\frac{(1+\lambda^2)}{2}(\inv{\pi\lambda(1+\lambda^2) u})\right)}\right)\Biggr] \Biggr\}^{\inv{2}}\\
&\sim& - \sqrt{ - \frac{2}{1+\lambda^2}\log(- 2\pi\lambda u \log(2\pi \lambda u))} \label{asym quantile SN lambda>0}\\
&\sim& - \sqrt{-\frac{2}{1+\lambda^2}\log u}, \label{asym quantile SN lambda>0 simplified}
\end{eqnarray}}
as $u\to 0^+$ for $\lambda>0$; and {\adb
\begin{eqnarray}
\notag g(u)&=& - \left\{ \log\left( \frac{\left(\inv{u}\sqrt{\frac{2}{\pi}}\right)^2}{\log\left(\inv{u}\sqrt{\frac{2}{\pi}}\right)^2}\right)+ O\left(\frac{\log\log \left(\inv{u}\sqrt{\frac{2}{\pi}}\right)^2}{\log \left(\inv{u}\sqrt{\frac{2}{\pi}}\right)^2}\right)\right\}^{\inv{2}}\\
&\sim& -\sqrt{-2\log\left(\frac{u}{2}\sqrt{-4\pi \log(\frac{u}{2}\sqrt{2\pi}})\right)} \label{asym quantile SN lambda<0}\\
&\sim& -\sqrt{-2\log u}, \label{asym quantile SN lambda<0 simplified}
\end{eqnarray}}
as $u\to 0^+$ for $\lambda<0$. 
Now set $y(u)$ 
 as the right-hand side of (\ref{asym quantile SN lambda>0}) and (\ref{asym quantile SN lambda<0})  in the respective cases $\lambda>0$ and $\lambda<0$. It is clear that $-y(u)$ is SVF as $u\to 0^+$ from (\ref{asym quantile SN lambda>0 simplified}) and (\ref{asym quantile SN lambda<0 simplified}), and since $y(u) \to -\infty$ as $u\to 0^+$, that $F_1(y(u))/u \to 1$, using (\ref{asym quantile SN lambda>0}), (\ref{asym quantile SN lambda<0}) and the right-hand side of (\ref{SN tail behaviour first order}).
\end{proof}

Notice (and compare with (\ref{asym quantile SN lambda<0})) that the asymptotic behaviour of the quantile function for the standard normal is 
\begin{align}
\Phi^{-1}(u) \sim &  -\sqrt{-2\log(u\sqrt{-4\pi\log u})} \label{asym quantile standard normal} \\
\sim & -\sqrt{-2\log u}, , \quad \text{as $u\to 0^+$.} \notag
\end{align}
(See \citet{fung2011} where the above methodology is used to obtain (\ref{asym quantile standard normal}).)

\section{Main result}
Similarly to the univariate cdf (and the corresponding quantile function), the rate of convergence to zero of the lower and upper tail dependence function depends heavily on whether $\theta>0$ or $\theta<0$. The results are summarised into the following theorem.
\begin{Theorem}
Let $\textbf{X} \sim SN_2(\bs{\theta}, R)$ with $\theta_1 = \theta_2 = \theta$. As $u\to 0^+$, 
\begin{enumerate}
\item[(a)] if $\theta>0$, {\adb
\begin{align}
\lambda_L(u) &\sim u^{\beta^2}\frac{\alpha^3}{\pi \lambda^4\beta(1+\beta^2)^2} \sqrt{\frac{2}{\pi}} (2\pi \lambda)^{1+\beta^2}
(\frac{1+\lambda^2}{2})^{\frac{3}{2}} \left[ -\log  u\right]^{\beta^2-\inv{2}}\label{rate of convergence SN theta>0}
\end{align}}
with $\lambda = \frac{\theta(1+\rho)}{\sqrt{1+\theta^2(1-\rho^2)}}$, $\alpha = \frac{\theta(1+\rho)}{\sqrt{1+2\theta^2(1+\rho)}}$ and $\beta = \sqrt{\frac{(1-\rho)(1+2\theta^2(1+\rho))}{1+\rho}}$;
\item[(b)] if $\theta<0$,
\begin{equation}
\lambda_L(u) \sim u^{\frac{1-\rho}{1+\rho}} \times \frac{1+\rho}{2} \sqrt{\frac{1+\rho}{1-\rho}} (-\pi \log u)^{-\frac{\rho}{1+\rho}}.\label{rate of convergence SN theta<0}
\end{equation}
\end{enumerate}
\end{Theorem}
\begin{proof}
The proof will be divided into two parts depends on whether $\theta>0$ or $\theta<0$. We will first consider the case $\theta>0$ and mixture representation forms the basis of this proof. 

For a given pair of $(\bosy{\theta},\Sigma)$ in (\ref{pdf:SN_n}), it has been shown in Azzalini and Capitanio (1999) that there is a pair of $(\bs{\alpha},\Psi)$ such that $\textbf{X}$ can be represented as a normal mean mixture by 
\begin{equation}
\textbf{X} \overset{d}{=} \bosy{\alpha}V+ \textbf{Z}, \label{normal mean mixing}
\end{equation}
where 
\begin{equation}
\bosy{\alpha} = \frac{R\bosy{\theta}}{\sqrt{1+\bosy{\theta}^TR\bosy{\theta}}} = \left(\begin{matrix} \frac{\theta(1+\rho)}{\sqrt{1+2\theta^2(1+\rho)}}, & \frac{\theta(1+\rho)}{\sqrt{1+2\theta^2(1+\rho)}}\end{matrix}\right)^T = (\alpha, \alpha)^T; \label{alpha}
\end{equation}
and 
\begin{equation*}
\Psi = R- (1+\bosy{\theta}^TR\bosy{\theta})^{-1}R\bosy{\theta}\bosy{\theta}^TR= \left(\begin{smallmatrix} \frac{1+\theta^2(1-\rho^2)}{1+2\theta^2(1+\rho)} & \frac{\rho - \theta^2(1-\rho^2)}{1+2\theta^2(1+\rho)} \\ \frac{\rho - \theta^2(1-\rho^2)}{1+2\theta^2(1+\rho)}  & \frac{1+\theta^2(1-\rho^2)}{1+2\theta^2(1+\rho)}\end{smallmatrix}\right) \\
= \left(\begin{smallmatrix} \frac{\alpha^2}{\lambda^2} & \frac{\rho - \theta^2(1-\rho^2)}{1+2\theta^2(1+\rho)} \\ \frac{\rho - \theta^2(1-\rho^2)}{1+2\theta^2(1+\rho)}  & \frac{\alpha^2}{\lambda^2}\end{smallmatrix}\right). 
\end{equation*}
Note that $\theta>0$ implies that $\alpha>0$. This parametrisation also satisfies the condition that $\Psi$ is symmetric and positive definite. $\bs{Z} \sim N(\bs{0}, \Psi)$ is the bivariate normal with $\bs{0}$ mean and covariance matrix $\Psi$, and $V\sim $ Half Normal with pdf 
\begin{eqnarray}
f_V(v) &=& \sqrt{\frac{2}{\pi}} e^{-\frac{v^2}{2}}, \quad \text{if $v>0$;} \quad = 0, \quad \text{otherwise.} \label{pdf:half normal}
\end{eqnarray}
That is $V = |W|$, with $W \sim N(0,1)$. $V$ is assumed to be distributed independently of $\textbf{Z}$. 
Obviously, when $\bosy{\theta}=\textbf{0}$ (and equivalently $\bosy{\alpha}=\textbf{0}$), we have the usual (symmetric) multivariate normal as special case for the distribution of $\textbf{X}$.

Continuing, when $\theta>0$ {\adb
\begin{eqnarray}
\notag && P(X_1\leq F_1^{-1}(u), X_2 \leq F_2^{-1}(u))\\
\notag&=& P(X_1 \leq x, X_2 \leq x), \quad \text{where $x = x(u) = F_1^{-1}(u) = F_2^{-1}(u)$}\\
\notag&=& E_V( P(\alpha V+ Z_1 \leq x, \, \alpha V+ Z_2 \leq x))\\
\notag&=& E_V\left(P\left(\frac{Z_1}{\alpha/\lambda} \leq \frac{x - \alpha V}{\alpha/\lambda} , \frac{Z_2}{\alpha/\lambda}  \leq \frac{x - \alpha V}{\alpha/\lambda} \right)\right)\\
\notag &=&  E_V\left(P\left(Z_1^* \leq \frac{x - \alpha V}{\alpha/\lambda} , Z_2^* \leq \frac{x - \alpha V}{\alpha/\lambda} \right)\right)\\
 &=& E_V\left(P\left( \max(Z_1^*, Z_2^*)  \leq \frac{x - \alpha V}{\alpha/\lambda} \right)\right) \label{SN max}
\end{eqnarray}}
where we define
\begin{equation*}
\textbf{Z}^* = (Z_1^*, Z_2^*)^T = \left(\begin{matrix} \frac{Z_1}{\alpha/\lambda}, & \frac{Z_2}{\alpha/\lambda}\end{matrix} \right)^T \sim N(\bs{0}, \left(\begin{smallmatrix} 1 & \frac{\rho - \theta^2(1-\rho^2)}{1+\theta^2(1-\rho^2)}\\ \frac{\rho - \theta^2(1-\rho^2)}{1+\theta^2(1-\rho^2)} & 1 \end{smallmatrix}\right)).
\end{equation*}
Using the results from Roberts (1966) and Loperfido (2002), we know that $Z_{(2)}^* = \max( Z_1^*, Z_2^*) \sim SN(\beta)$ i.e. a univariate skew normal distribution with skewness  
\begin{equation}
\beta = \sqrt{\frac{1- \frac{\rho - \theta^2(1-\rho^2)}{1+\theta^2(1-\rho^2)}}{1+ \frac{\rho - \theta^2(1-\rho)}{1+\theta^2(1-\rho^2)}}} = \sqrt{\frac{(1-\rho)(1+2\theta^2(1+\rho))}{(1+\rho)}}. \label{eqn: beta}
\end{equation}
If we combine this fact with the Capitanio bounds in (\ref{SN tail behaviour lambda>0}), (\ref{SN max}) becomes
{\adb
\begin{eqnarray}
\notag && \int^{\infty}_{0} \inv{\pi} e^{-\inv{2}(1+\beta^2)\left(\frac{x-\alpha v}{\alpha/\lambda}\right)^2}\times \inv{\beta(1+\beta^2)}\left(\frac{x-\alpha v}{\alpha/\lambda}\right)^{-2} f_V(v)\,dv \\
\notag &-& \int^{\infty}_{0} \inv{\pi} e^{-\inv{2}(1+\beta^2)\left(\frac{x-\alpha v}{\alpha/\lambda}\right)^2}\left(\frac{2}{\beta(1+\beta^2)}+\inv{\beta^3(1+\beta^2)}\right)\left(\frac{x-\alpha v}{\alpha/\lambda}\right)^{-4} f_V(v)\,dv \\
\notag &<& E_V\left(P\left( Z_{(2)}^* \leq \frac{x - \alpha V}{\alpha/\lambda} \right)\right)\\
&<&  \int^{\infty}_{0} \inv{\pi} e^{-\inv{2}(1+\beta^2)\left(\frac{x-\alpha v}{\alpha/\lambda}\right)^2}\inv{\beta(1+\beta^2)}\left(\frac{x-\alpha v}{\alpha/\lambda}\right)^{-2} f_V(v)\,dv,  \label{eqn: the inequality}
\end{eqnarray}}
which suggests that we need to compare the upper bound in (\ref{eqn: the inequality}) with 
\begin{equation*}
|x|^{-3} e^{-\frac{\lambda^2}{2\alpha^2}(1+\beta^2)x^2},
\end{equation*}
and we will consider 
\begin{eqnarray*}
&& \frac{\int^{\infty}_{0} \inv{\pi} e^{-\inv{2}(1+\beta^2)\left(\frac{x-\alpha v}{\alpha/\lambda}\right)^2}\times \inv{\beta(1+\beta^2)}\left(\frac{x-\alpha v}{\alpha/\lambda}\right)^{-2} f_V(v)\,dv }{|x|^{-3} e^{-\frac{\lambda^2}{2\alpha^2}(1+\beta^2)x^2}}\\
&=& |x| \int^{\infty}_{0} \inv{\pi} e^{-\frac{\lambda^2}{2\alpha^2}(1+\beta^2)(-2\alpha x v +\alpha^2 v^2)}\times \frac{\alpha^2}{\lambda^2\beta(1+\beta^2)} \inv{(1+\alpha\frac{v}{|x|})^2}f_V(v)\,dv
\end{eqnarray*}
Applying integration by parts: {\adb
\begin{align}
\notag =& \frac{\alpha^2|x|}{\pi \lambda^2\beta(1+\beta^2)}\bigg\{ 
\left[ \frac{\alpha}{\lambda^2(1+\beta^2)x} e^{\frac{\lambda^2}{\alpha}(1+\beta^2)xv} \times \frac{e^{-\lambda^2(1+\beta^2)(\frac{v^2}{2})} f_V(v)}{(1+\alpha\frac{v}{|x|})^2}\right]^{\infty}_{0}\\
\notag  & \quad - \int^{\infty}_{0} \frac{\alpha}{\lambda^2(1+\beta^2) x}e^{\frac{\lambda^2}{\alpha}(1+\beta^2)xv} \times \frac{d}{dv}\left(\frac{e^{-\lambda^2(1+\beta^2)(\frac{ v^2}{2})} f_V(v)}{(1+\alpha\frac{v}{|x|})^2}\right)\,dv\bigg\}\\
\notag =& \frac{\alpha^2}{\pi \lambda^2\beta(1+\beta^2)}\bigg\{ 
 \frac{\alpha f_V(0)}{\lambda^2(1+\beta^2)} + \int^{\infty}_{0} \frac{\alpha}{\lambda^2(1+\beta^2)} e^{\frac{\lambda^2}{\alpha}(1+\beta^2)xv} \\
& \quad \times \frac{d}{dv}\left(\frac{e^{- \lambda^2(1+\beta^2)(\frac{ v^2}{2})} f_V(v)}{(1+\alpha\frac{v}{|x|})^2}\right)\,dv\bigg\} \label{being dominated 1}
\end{align}}
As {\adb
\begin{eqnarray*}
&& \left| \frac{d}{dv}\left(\frac{e^{-\lambda^2(1+\beta^2)(\frac{ v^2}{2})} f_V(v)}{(1+\alpha\frac{v}{|x|})^2}\right)\right| \\
&=& \Bigg| \frac{-2}{(1+\alpha \frac{v}{|x|})^3}\frac{\alpha}{|x|}e^{-\lambda^2(1+\beta^2)(\frac{ v^2}{2})} f_V(v) + \inv{(1+\alpha\frac{v}{|x|})^2}\left[ - \lambda^2(1+\beta^2)v\right]\\
&& \quad \times  e^{-\lambda^2(1+\beta^2)(\frac{v^2}{2})} f_V(v)+ \frac{e^{-\lambda^2(1+\beta^2)(\frac{v^2}{2})}}{(1+\alpha\frac{v}{|x|})^2} \frac{d}{dv} f_V(v)\Bigg| \\
&\leq& e^{-\lambda^2(1+\beta^2)(\frac{v^2}{2})}\left(2\alpha f_V(v) + \lambda^2(1+\beta^2)v f_V(v) + \bigg|\frac{d}{dv} f_V(v)\bigg|\right) < \infty,
\end{eqnarray*}}
and $\frac{d}{dv} f_V(v) = \sqrt{\frac{2}{\pi}} (-v)e^{-v^2/2},$
we only need 
\begin{equation*}
\int^{\infty}_{0} ve^{-\xi v^2}\,dv <\infty
\end{equation*}
for some $\xi>0$ to have dominated convergence in (\ref{being dominated 1}) and the condition is obviously true. Thus, {\adb
\begin{eqnarray*}
&& \lim_{x\to -\infty} \frac{\int^{\infty}_{0} \inv{\pi} e^{-\inv{2}(1+\beta^2)\left(\frac{x-\alpha v}{\alpha/\lambda}\right)^2}\times \inv{\beta(1+\beta^2)}\left(\frac{x-\alpha v}{\alpha/\lambda}\right)^{-2} f_V(v)\,dv }{|x|^{-3}e^{-\frac{\lambda^2}{2\alpha^2}(1+\beta^2)x^2}}= \frac{\alpha^3f_V(0)}{\pi \lambda^4\beta(1+\beta^2)^2},
\end{eqnarray*}}
which also implies that {\adb
\begin{align}
\notag & \int^{\infty}_{0} \inv{\pi} e^{-\inv{2}(1+\beta^2)\left(\frac{x-\alpha v}{\alpha/\lambda}\right)^2}\times \inv{\beta(1+\beta^2)}\left(\frac{x-\alpha v}{\alpha/\lambda}\right)^{-2} f_V(v)\,dv\\
\sim& \frac{\alpha^3}{\pi \lambda^4\beta(1+\beta^2)^2}\sqrt{\frac{2}{\pi}} |x|^{-3}e^{-\frac{\lambda^2}{2\alpha^2}(1+\beta^2)x^2}.  \label{eqn: the inequality dominating term}
\end{align}}
\indent As the first term in the lower bound is the same as the upper bound in (\ref{eqn: the inequality}), we will now consider the higher order term in the lower bound in the form 
\begin{align*}
& \frac{\int^{\infty}_{0} \inv{\pi} e^{-\inv{2}(1+\beta^2)\left(\frac{x-\alpha v}{\alpha/\lambda}\right)^2}\times \left(\frac{2}{\beta(1+\beta^2)}+\inv{\beta^3(1+\beta^2)}\right)\left(\frac{x-\alpha v}{\alpha/\lambda}\right)^{-4} f_V(v)\,dv}{|x|^{-5} e^{-\frac{\lambda^2}{2\alpha^2}(1+\beta^2)x^2}}\\
=& |x| \int^{\infty}_{0} \inv{\pi} e^{-\frac{\lambda^2}{2\alpha^2}(1+\beta^2)(-2\alpha v x +\alpha^2 v^2)}\times \left(\frac{2}{\beta(1+\beta^2)}+\inv{\beta^3(1+\beta^2)}\right)\left(\frac{\alpha^4}{\lambda^4}\right)\frac{f_V(v)}{(1+\alpha \frac{v}{|x|})^4}\,dv\\
\to&\frac{\alpha^5}{\pi\lambda^6(1+\beta^2)}\left(\frac{2}{\beta(1+\beta^2)}+\inv{\beta^3(1+\beta^2)}\right)f_V(0)
\end{align*}
after integration by parts and dominated convergence again as $x\to -\infty$.
Thus {\adb
\begin{eqnarray}
\notag && \int^{\infty}_{0} \inv{\pi} e^{-\inv{2}(1+\beta^2)\left(\frac{x-\alpha v}{\alpha/\lambda}\right)^2}\times \left(\frac{2}{\beta(1+\beta^2)}+\inv{\beta^3(1+\beta^2)}\right)\left(\frac{x-\alpha v}{\alpha/\lambda}\right)^{-4} f_V(v)\,dv\\
&=& O\left(|x|^{-5}e^{-\frac{\lambda^2}{2\alpha^2}(1+\beta^2)x^2}\right) \label{eqn: the inequality correcting term}
\end{eqnarray}}
as $x\to -\infty$. By combining (\ref{eqn: the inequality}), (\ref{eqn: the inequality dominating term}) and (\ref{eqn: the inequality correcting term}), we have 
\begin{equation*}
P(X_1 \leq x, X_2 \leq x) \sim \frac{\alpha^3}{\pi \lambda^4\beta(1+\beta^2)^2} \sqrt{\frac{2}{\pi}} |x|^{-3} e^{-\frac{\lambda^2}{2\alpha^2}(1+\beta^2)x^2},
\end{equation*}
as $x\to -\infty$ and by substituting in (\ref{asym quantile SN lambda>0}) we have {\adb
\begin{eqnarray*}
&& P(X_1\leq F_1^{-1}(u), X_2\leq F_2^{-1}(u)) \\
&\sim& \frac{\alpha^3}{\pi \lambda^4\beta(1+\beta^2)^2} \sqrt{\frac{2}{\pi}} \left(- \frac{2}{1+\lambda^2}\log(- 2\pi\lambda u \log(2\pi \lambda u))\right)^{-\frac{3}{2}}\\
&& \quad \times e^{-\frac{\lambda^2}{2\alpha^2}(1+\beta^2)\left(-\frac{2}{1+\lambda^2}\log(- 2\pi\lambda u \log(2\pi \lambda u))\right)}\\
&\sim& \frac{\alpha^3}{\pi \lambda^4\beta(1+\beta^2)^2}  \sqrt{\frac{2}{\pi}} (2\pi \lambda)^{\frac{\lambda^2(1+\beta^2)}{\alpha^2(1+\lambda^2)}}
(\frac{1+\lambda^2}{2})^{\frac{3}{2}} 
\left[ -\log  u\right]^{\frac{\lambda^2(1+\beta^2)}{\alpha^2(1+\lambda^2)}-\frac{3}{2}}u^{\frac{\lambda^2(1+\beta^2)}{\alpha^2(1+\lambda^2)}}\\
&=& \frac{\alpha^3}{\pi \lambda^4\beta(1+\beta^2)^2} \sqrt{\frac{2}{\pi}} (2\pi \lambda)^{1+\beta^2}
(\frac{1+\lambda^2}{2})^{\frac{3}{2}} \left[ -\log  u\right]^{\beta^2-\inv{2}}  \times u^{1+\beta^2}.
\end{eqnarray*}}
as
\begin{eqnarray*}
\alpha^2(1+\lambda^2) &=& \frac{\theta^2(1+\rho)^2}{1+2\theta^2(1+\rho)}\left(1+\frac{\theta^2(1+\rho^2)}{1+\theta^2(1-\rho^2)}\right) \\
&=& \frac{\theta^2(1+\rho)^2}{1+2\theta^2(1+\rho)}\left(\frac{1+2\theta^2(1+\rho)}{1+\theta^2(1-\rho^2)}\right) = \lambda^2 
\end{eqnarray*}
since $\lambda = \frac{\theta(1+\rho)}{\sqrt{1+\theta^2(1-\rho^2)}}$ from (\ref{lambda}) and $\alpha = \frac{\theta(1+\rho)}{\sqrt{1+2\theta^2(1+\rho)}}$ from (\ref{alpha}). Finally,  from (\ref{defn:tail dependence}){\adb
\begin{eqnarray*}
&& P(X_1\leq F_1^{-1}(u) | X_2\leq F_2^{-1}(u)) = \frac{P(X_1\leq F_1^{-1}(u), X_2\leq F_2^{-1}(u))}{u} \\
\notag &\sim& \frac{\alpha^3}{\pi \lambda^4\beta(1+\beta^2)^2} \sqrt{\frac{2}{\pi}} (2\pi \lambda)^{1+\beta^2}
(\frac{1+\lambda^2}{2})^{\frac{3}{2}} \left[ -\log  u\right]^{\beta^2-\inv{2}} \times u^{\beta^2},
\end{eqnarray*}}
which is (\ref{rate of convergence SN theta>0}) and part (a) of the proof is now completed.

Next, we consider the case $\theta<0$. In this part of the proof, we proceed by noting that from (\ref{defn:tail dependence 3}) that 
$
C(u,u)= \int_0^u \frac{dC(z,z)}{dz}dz $, so that if ${\frac{dC(z,z)}{dz}}= z^{\tau}L(z),
\tau > 0 $ where $L(z)$ is a slowly varying function as $z \to
0^+$, then by (applying with suitable transformation to regular
variation at $0$) a result of de Haan (see Seneta (1976), p. 87), we
obtain
\begin{equation}
\frac{C(u,u)}{u} = \inv{u}\int_0^u {\frac{dC(z,z)}{dz}}dz \sim  {\frac{u^{\tau}L(u)}{\tau + 1}}, \, \, u \to 0^+ .\label{defn:tail dependence13}\end{equation}
Therefore, it is sufficient for us to find a value of $\tau > 0$
which satisfies
$\frac{d C(u,u)}{du} = u^{\tau} L(u)$, for some slowly varying function $L(u)$, as $u\rightarrow 0^+$, so
that (\ref{defn:tail dependence13}) holds.

From (\ref{defn:tail dependence}), using L'H\^{o}pital's rule and some well
established basic properties of the derivative of copulas (see Nelsen (2006),
pp.13, 41), we have
\begin{eqnarray}
\notag \frac{d C(u,u)}{du}&=& P(X_2 \leq F_2^{-1}(u)|X_1 =F_1^{-1}(u))
+P(X_1 \leq F_1^{-1}(u)|X_2 =F_2^{-1}(u))\\
&=& 2 P(X_2 \leq x |X_1 =x), \label{L'Hopital's rule}
\end{eqnarray}
by letting $x = F_1^{-1}(u) = F_2^{-1}(u)$. By using (\ref{pdf:SN_n}) and (\ref{pdf:SN_1}), the last expression can be written as {\adb
\begin{eqnarray*}
&& 2 P(X_2\leq  x | X_1 = x)  = 2 \int^{x}_{-\infty} \inv{\sqrt{2\pi(1-\rho^2)}}e^{-\frac{(x_2-\rho x)^2}{2(1-\rho^2)}}\frac{\Phi(\theta x + \theta x_2)}{\Phi(\lambda x)}\,dx_2\\
&=& 2 \int^{\sqrt{\frac{1-\rho}{1+\rho}}x}_{-\infty} \inv{\sqrt{2\pi}}e^{-\inv{2}z^2}\frac{\Phi(\theta\sqrt{1-\rho^2}z+\theta(1+\rho)x)}{\Phi(\lambda x)}\,dz, \quad \text{by letting $z = \frac{x_2-\rho x}{\sqrt{1-\rho^2}}$.}
\end{eqnarray*}}
Since $\theta<0$, we have $\lambda = \frac{\theta(1+\rho)}{\sqrt{1+\theta^2(1-\rho^2)}}<0,$ which implies that $\Phi(\lambda x) \to 1$ as $x\to -\infty$, that is, when $u\to 0^+$. Moreover, {\adb
\begin{align*}
 \quad -\infty < z < \sqrt{\frac{1-\rho}{1+\rho}} x \quad &\Rightarrow \quad \theta(1-\rho)x < \theta\sqrt{1-\rho^2}z < \infty, \quad \text{as $\theta<0$;}\\
 &\Rightarrow \quad 2\theta x < \theta\sqrt{1-\rho^2}z+\theta(1+\rho)x < \infty\\
 &\Rightarrow \quad \Phi(2\theta x) < \Phi(\theta\sqrt{1-\rho^2}z+\theta(1+\rho)x) < 1.
\end{align*}}
As a result, 
\begin{eqnarray*}
2\int^{\sqrt{\frac{1-\rho}{1+\rho}}x}_{-\infty} \inv{\sqrt{2\pi}}e^{-\inv{2}z^2}\frac{\Phi(2\theta x)}{\Phi(\lambda x)}\,dz < 2 P(X_2\leq x | X_1 = x)  < 2\int^{\sqrt{\frac{1-\rho}{1+\rho}}x}_{-\infty} \inv{\sqrt{2\pi}}\frac{e^{-\inv{2}z^2}}{\Phi(\lambda x)}\,dz.
\end{eqnarray*}
Since both $\Phi(2\theta x)$ and $\Phi(\lambda x) \to 1$ as $x\to -\infty$ we have {\adb
\begin{eqnarray}
\notag && 2 P(X_2\leq x | X_1 = x)  \sim  2\int^{\sqrt{\frac{1-\rho}{1+\rho}}x}_{-\infty} \inv{\sqrt{2\pi}}e^{-\inv{2}z^2}\,dz = 2 \Phi(\sqrt{\frac{1-\rho}{1+\rho}}x)\\
\notag \text{so} && \quad 2 P(X_2\leq  F_2^{-1}(u)| X_1 = F_1^{-1}(u)) \sim  2 \Phi(\sqrt{\frac{1-\rho}{1+\rho}} F_1^{-1}(u)),
\end{eqnarray}}
since $F_1^{-1}(u) \to -\infty$ as $u\to 0^+$. 
Since $F_1^{-1}(u)\sim y(u)$ as $u\to 0^+$ from Theorem 1 and using (\ref{normal cdf expansion}), we can prove $2 \Phi(\sqrt{\frac{1-\rho}{1+\rho}} F_1^{-1}(u)) \sim 2 \Phi(\sqrt{\frac{1-\rho}{1+\rho}} y(u))$ as $u\to 0^+$
by showing $[F_1^{-1}(u)]^2 - [y(u)]^2 \to 0$,  as $u\to 0^+$.
This is in turn equivalent to showing  
\begin{equation*}
e^{[F_1^{-1}(F_1(z))]^2 - [y(F_1(z))]^2} = e^{z^2 + 2\log(\frac{F_1(z)}{2} \sqrt{-4\pi \log(\frac{F_1(z)}{2}\sqrt{2\pi}}} \to 1, \quad \text{as $z\to -\infty$}. 
\end{equation*}
Then 
\begin{eqnarray*}
&& e^{z^2 + 2\log(\frac{F_1(z)}{2} \sqrt{-4\pi \log(\frac{F_1(z)}{2}\sqrt{2\pi})})} \\
&\sim& 
e^{z^2} \times (\inv{\sqrt{2\pi}} |z|^{-1} e^{-z^2/2})^2\times (-4\pi \log(|z|^{-1} e^{-z^2/2})),\quad \text{by (\ref{SN tail behaviour first order})}\\
&=& 1+ \frac{4\pi \log |z|}{2\pi z^2} \to 1,\quad \text{as $z \to -\infty$}.
\end{eqnarray*}
This implies that 
\begin{equation*}
2 P(X_2\leq  F_2^{-1}(u)| X_1 = F_1^{-1}(u)) \sim  2 \Phi(\sqrt{\frac{1-\rho}{1+\rho}} y(u)) \sim \sqrt{\frac{1+\rho}{1-\rho}} (-\pi \log u)^{-\frac{\rho}{1+\rho}} u^{\frac{1-\rho}{1+\rho}}
\end{equation*}
as $u\to 0^+$, by using (\ref{normal cdf expansion}) and (\ref{asym quantile SN lambda<0}). Finally, by using (\ref{defn:tail dependence13}) we have 
\begin{equation*}
P(X_2\leq F_2^{-1}(u) | X_1\le F_1^{-1}(u)) \sim \frac{1+\rho}{2} \sqrt{\frac{1+\rho}{1-\rho}} (-\pi \log u)^{-\frac{\rho}{1+\rho}} u^{\frac{1-\rho}{1+\rho}}, \quad \text{as $u\to 0^+$}.
\end{equation*}
\end{proof}
The theorem shows that when $\theta<0$ there is minimal difference between symmetric and skew normal in terms of the intermediate tail dependence as they share the same regular varying index which is $\frac{1-\rho}{1+\rho}$. On the other hand, when $\theta>0$, it has a larger regular varying index by a factor of $(1+2\theta^2(1+\rho))$ when compared to the normal case and therefore skew normal has smaller intermediate tail dependence than the normal in the lower tail. 

We shall finish the paper by briefly discussing the corresponding result for the upper tail dependence. 
\begin{Corollary}Let $\textbf{X} \sim SN_2(\bs{\theta}, R)$ with $\theta_1 = \theta_2 = \theta$. In self evident notation and as $u\to 1^-$, 
\begin{enumerate}
\item[(a)] if $\theta< 0$ (so $-\theta = |\theta|>0$), {\adb
\begin{align*}
\notag & \lambda_U(u)  = P(X_1 \geq F_1^{-1}(u) | X_2 \geq F_2^{-1}(u))\\\
&\sim (1-u)^{\beta^2}\frac{\alpha^3}{\pi \lambda^4\beta(1+\beta^2)^2} \sqrt{\frac{2}{\pi}} (2\pi \lambda)^{1+\beta^2}
(\frac{1+\lambda^2}{2})^{\frac{3}{2}} \left[ -\log (1-u)\right]^{\beta^2-\inv{2}}
\end{align*}}
with $\lambda = \frac{|\theta|(1+\rho)}{\sqrt{1+\theta^2(1-\rho^2)}}$, $\alpha = \frac{|\theta|(1+\rho)}{\sqrt{1+2\theta^2(1+\rho)}}$ and $\beta = \sqrt{\frac{(1-\rho)(1+2\theta^2(1+\rho))}{1+\rho}}$;
\item[(b)] if $\theta> 0$ (so $-\theta<0$),
\begin{equation*}
\lambda_U(u) \sim (1-u)^{\frac{1-\rho}{1+\rho}} \times \frac{1+\rho}{2} \sqrt{\frac{1+\rho}{1-\rho}} (-\pi \log(1- u))^{-\frac{\rho}{1+\rho}}.
\end{equation*}
\end{enumerate}
\end{Corollary}
\begin{proof}
If $\textbf{Y}  = (Y_1, Y_2)^T \overset{d}{=} - \textbf{X}$ where $\textbf{X} = (X_1, X_2)^T$ with continuous marginal distributions, then from Lemma 1 of Fung and Seneta (2014) (with self-evident notation), we have $\lambda_U^{\textbf{X}}(u) = \lambda_L^{\textbf{Y}}(1-u)$. The proof is completed by noting that when $\textbf{X} \sim SN_2(\bs{\theta}, R)$, we have $\textbf{Y} = -\textbf{X} \sim SN_2(-\bs{\theta}, R)$ from Azzalini and Dalla Valle (1996) and applying Theorem 2.
\end{proof}


\begin{thebibliography}{9}

\bibitem[Azzalini and Capitanio(1999)]{azza1999}
Azzalini, A. and Capitanio, A., 1999. Statistical applications of the multivariate skew normal distribution. \emph{Journal of the Royal Statistical Society: Series B}, \textbf{61}, 579--602. 
\bibitem[Azzalini and Dalla Valle(1996)]{azza1996}
Azzalini, A. and Dalla Valle, A., 1996. The multivariate skew-normal distribution.
\emph{Biometrika} \textbf{83}, 715--726.



\bibitem[Bortot(2010)]{bort2010}
Bortot, P., 2010. Tail dependence in bivariate skew-normal and skew-$t$ distribution. http://www2.stat.unibo.it/bortot/ricerca/paper-sn-2.pdf


\bibitem[Capitanio(2010)]{capi2010}
Capitanio, A., 2010. On the approximation of the tail probability of the scalar skew-normal distribution. \emph{Metron} \textbf{68}, 299--308.


\bibitem[Chicheportiche and Bouchaud(2012)]{chic2012}
Chicheportiche, R. and Bouchaud, J.-P., 2012. The joint distribution of stock returns is not elliptical. \emph{International Journal of Theoretical and Applied Finance}, \textbf{15}, 1250019 (23 pages).

\bibitem[Corless Corless \emph{et al.}(1996)]{corl1996}
Corless, R.M., Gonnet, G.H., Hare, D.E.G, Jeffery, D.J. and Knuth, D.E., 1996. On the Lambert $W$ function. \emph{Advances in Computational Mathematics}, \textbf{5}, 329--359.


\bibitem[Feller(1968)]{fell1968}
Feller, W., 1968. \emph{An Introduction to Probability Theory and
Its Applications} \textbf{1}. Wiley, New York.

\bibitem[Fung and Seneta(2011)]{fung2011}
Fung, T. and Seneta, E., 2011. The  bivariate normal copula function is  regularly varying. \emph{Statistics and Probability Letters} \textbf{81}, 1670--1676.
\bibitem[Fung and Seneta(2014)]{fung2014}
Fung, T. and Seneta, E., 2014. Convergence rate to a lower tail dependence coefficient of a skew-$t$ distribution. \emph{Journal of Multivariate Analysis}, \textbf{128}, 62--72.



\bibitem[Hua and Joe(2011)]{hua2011}
Hua, L. and Joe, H., 2011. Tail order and intermediate tail dependence of multivariate copulas. \emph{Journal of Multivariate Analysis} \textbf{102}, 1454--1471.



\bibitem[Ledford and Tawn(1997)]{ledf1997}
Ledford, A.W. and Tawn, J.A., 1997. Modelling dependence with joint tail regions. \emph{Journal of the Royal Statistical Society: Series B} \textbf{59}, 475--499.

\bibitem[Loperfido(2002)]{lope2002}
Loperfido, N., 2002. Statistical implications of selectively reported inferential results. \emph{Statistics and Probability Letters} \textbf{56}, 13--22.

\bibitem[Lysenko, Roy and Waeber(2009)]{lyse2009}
Lysenko, N., Roy, P. and Waeber, R., 2009. Multivariate extremes of generalised skew-normal distributions. \emph{Statistics and Probability Letters} \textbf{79}, 525--533.
\bibitem[Manner and Segers(2011)]{mann2011}
Manner, H. and Segers, J., 2011. Tails of correlation mixtures of elliptical copulas. \emph{Insurance: Mathematics and Economics}, \textbf{48}, 153--160.

\bibitem[Nelsen(2006)]{nels2006}
Nelsen, R.B., 2006. \emph{An Introduction to Copulas}.  2nd ed. New York: Springer. 

\bibitem[Padoan(2011)]{pado2011}
Padoan, S.A., 2011. Multivariate extreme models based on underlying skew $t$ and skew-normal distributions. \emph{Journal of Multivariate Analysis} \textbf{102}, 977--991. 
\bibitem[Ramos and Ledford(2009)]{ramo2009}
Ramos, A. and Ledford, A., 2009. A new class of models for bivariate joint tails. \emph{Journal of the Royal Statistical Society: Series B}, \textbf{71}, 219--241.
\bibitem[Roberts(1966)]{robe1966}
Roberts, C., 1966. A correlation model useful in the study of twins. \emph{Journal of the American Statistical Association} \textbf{61}, 1184--1190.
\bibitem[Seneta(1976)]{sene1976}
Seneta, E., 1976. \emph{Regularly Varying Function. Lecture Notes in Mathematics} \textbf{508}. Berlin: Springer.






\end{thebibliography}
\end{document}